\documentclass[10 pt]{amsart}
 \usepackage{amsmath}
\usepackage{amssymb}
\usepackage{amsfonts}
 \usepackage{eucal}
      \makeatletter
     \def\section{\@startsection{section}{1}%
     \z@{.7\linespacing\@plus\linespacing}{.5\linespacing}%
     {\bfseries
     \centering
     }}
     \def\@secnumfont{\bfseries}
     \makeatother
\newtheorem{theorem}{Theorem}[section]

\newtheorem{proposition}[theorem]{Proposition}

\theoremstyle{definition}

\theoremstyle{remark}
\newtheorem{remark}[theorem]{Remark}

\theoremstyle{remarks}

\numberwithin{equation}{section}

\def \a{{\alpha}}

\def \d{{\delta}}
\def \e{{\varepsilon}}
\def \g{{\gamma}}

\def \l{{\lambda}}
\def \t{{\theta}}
\def \o{{\omega}}

\def \m{{\mu}}

\def \qq{{\qquad}}
\def \R{{\bf R}}
\def \dd{{\rm d}}

\def\R{{\mathbb R}}

 \scrollmode
  \font\sevenrm= cmr10 at 7,3 pt
 \font\eightrm= cmr10 at 8 pt

\def\ddate {\sevenrm \ifcase\month\or January\or
February\or March\or April\or May\or June\or July\or
August\or September\or October\or November\or December\fi\! {\the\day}, \!{\sevenrm\the\year}}

     \title[  $\hbox{\eightrm On the equation}$\   $n_1n_2=n_3n_4$\ $\hbox{\eightrm restricted\ to\ FC\ sets}$]
   { $\bf{On\ the\ equation}$  $ \boldsymbol{n_1n_2=n_3n_4}$\\  $\bf{restricted\ to\ factor\ closed \ sets}$}
    
\begin{document}
  \author{Sanying Shi}
 \address{School  of Math., Hefei University of Technology, Hefei 230009, P. R. China.}
 \email{\tt vera123${}_{-}$99@hotmail.com}
  \author{Michel  Weber}
 \address{IRMA, UMR 7501,   7  rue  Ren\'e Descartes, 67084 Strasbourg Cedex, France.}
\email{michel.weber@math.unistra.fr}


  \begin{abstract}  
   We study the number of solutions $N(B,F)$ of the diophantine equation $n_1n_2=n_3n_4$, where $1\le n_1\le B$, 
$1\le n_3\le B$, $n_2, n_4\in F$ and $F\subset [1,B]$ is a factor closed set.  We study more particularly the case when $F= \big\{m=p_1^{\e_1}\ldots
p_k^{\e_k}, \e_j\in \{0,1\}, 1\le j\le k\big\}$, 
   $p_1,\ldots,p_k$ being distinct prime numbers.

\end{abstract}

 \maketitle

\section{\bf Introduction-Results.} \label{s1}

 The number  of solutions $N(B)$ of the diophantine equation
  $n_1n_{2} =n_{3} n_{4}$  with unknowns verifying $1\le n_i\le B$,  $1\le i\le 4$
    was studied in Shi \cite{S},  where the following very precise estimate (see Theorem 1), improving  upon previous similar estimate by Ayyad, Cochrane and Zheng \cite{ACZ} is established,
 \begin{equation}\label{e.2}N(B) = \frac{12}{\pi^2}B^2\log B + CB^2 + \mathcal O \big( B^{\frac{547}{416}+\e}\big),\end{equation}
where
$C= \frac{36}{\pi^2}\g_0 -2\g_1+\frac{12}{\pi^2}\g_2-\frac{24}{\pi^2}-1=0.511317447...$  and
$\g_1=\frac{36}{\pi^2}\zeta'(2)$, $\g_2= \frac{3}{2} -\g_0-\frac{\pi^2}{12}$,
$\g_0$ being Euler's constant. The proof connects the evaluation of $N(B)$ with the error term in the Dirichlet divisor problem, and only involves elementary arithmetic. Some more simplifications were provided by  Liu and Zhai \cite{LZ}.

\vskip 3 pt In this work, we  study the following variant of the initial equation. Let
$F$ be a finite set of integers.  We assume that
$F$ is   factor closed, in short $F$ is an FC set. By this we mean that $d|k\Rightarrow d\in F $, for all $k \in F$; in particular $1\in F$. By definition, this notion extends to sets $F$ formed with {not necessarily distinct} elements, for instance the set $\mathcal M_k(B)$, $B\ge 1$, of all possible products $m_1\ldots m_{k}$ obtained by taking  $1\le m_i\le B$, $1\le i\le k$.
We refer to Haukkanen,  Wang  and   Sillanp\"a\"a \cite{HWS} (see also references therein) concerning this
notion and extensions, also  Weber  \cite{W1}. Typical examples of FC sets are naturally intervals $[1,B]$, the set of divisors of an integer, the set of squarefree integers less than $B$, the multiplicative semi-group  generated by a given set of integers, the trace $F\cap [1,B]$ of an FC set $F$. \vskip 2 pt
Let $N(B,F)$ denote the number of integers solutions of   the restricted equation
   \begin{eqnarray}\label{e2} n_1n_2 =n_3n_4,
    \end{eqnarray}  where the unknowns satisfy $1\le n_1\le B$, $1\le n_3\le B$, $n_2, n_4\in F$. 
\vskip 2 pt
    It is of interest to  observe that the initial equation 
is just a particular case of equation \eqref{e2}. This can be generalized. Let $k\ge 1$ and let $N_{k}(B)$ denote the number of  solutions  of
the equation  \begin{equation}\label{e.k+1}n_1\ldots n_{k+1} =n_{k+2}\ldots n_{2(k+1)},
\end{equation}
  with unknowns verifying $1\le n_i\le B$,  $1\le i\le 2(k+1)$.
One sees that solving equation \eqref{e2} amounts to solving equation \eqref{e.k+1} with $F=\mathcal M_k(B)$.

 We also consider equation \eqref{e2} with the constraint that all unknowns must belong to $F$, and denote by $N(F,F)$ the corresponding number of solutions. 
 \vskip 2  pt 
 We use the approach developed in \cite{S} to establish the following results. Given two numbers $n$ and $m$, let $n\vee m= \max (n,m)$.

\begin{proposition} \label{n}For any FC set $F$ and any integer $B\ge 1$, we have
  \begin{eqnarray*}       N(B,F)&=&\sum_{ n, m\in F\atop \gcd(n,m)=1}\Big( \sum_{  \m, \nu\in F\atop \frac{\nu}{\m}= \frac{m}{n}} 1\Big)\   \Big\lfloor \frac{B}{n\vee m}\Big\rfloor.
    \end{eqnarray*}
      If $F\subset [1,B]$, then
   \begin{eqnarray*}
    N(B,F)& \le & 2 B^2  \sum_{n \in F\atop n\neq 1}\frac{1}{n}  +B^2,
    \cr N(B,F)& \ge  & \sum_{\m,\nu\in F\atop } \gcd(\m,\nu),
 \cr
   N(F, F )& =&\sum_{ n, m\in F\atop \gcd(n,m)=1}\Big( \sum_{  n_2, n_4\in F\atop \frac{n_4}{n_2}= \frac{m}{n}} 1\Big)^2  .
    \end{eqnarray*} \end{proposition}
 \begin{remark} For the case $F=[1,B]$, we recover that $N(B)\le C B^2 \log B$. 
 We also recall that (\cite[Exercise 59]{T})
 \begin{eqnarray} 
  \sum_{1\le \m,\nu\le B} \gcd(\m,\nu)= \frac{6}{\pi^2}   B^2\log B+ \mathcal O(B^2). 
 \end{eqnarray}
\end{remark}
 \vskip 4  pt 
     We deduce
    
     \begin{theorem}\label{nf} Suppose that $F=\{m_1m_2: 1\le m_1,m_2\le [B^{\a}]\}$ with $\a \in [0,1/2]$ . We have
    \begin{eqnarray*}
    N(B, F )\le \Big(1+\frac {12}{\pi^2}+2\sum_{  B^{\a} <m\le B^{2\a}\atop    m\in F}\frac 1{m}\Big) B^{1+2\a}.
     \end{eqnarray*}
  \end{theorem}
  Let us consider the following  typical example of a FC set. Let
\begin{equation} \label{F} F= \Big\{m=p_1^{\e_1}\ldots p_k^{\e_k}\,|\, \e_j\in \{0,1\}, 1\le j\le k\Big\}, 
\end{equation}
where  $p_1<\ldots<p_k$ are prime numbers.
Recall that $\o(n)=\#\{p\,\hbox{prime}:p|n\}$ is the prime divisor function, and that $\o(1)=0$.

\begin{theorem}\label{np}
For any positive integer $B$ and $FC$ set  of the type \eqref{F},
 \begin{align*}      
      N(B,F)    \le    C\,B2^k\prod_{i=1}^k &\big(1+\frac{1}{4
p_i}\big) \min\big(B \,,\,  \big(\frac{5}{4}\big)^{k  } \big)   
  \cr    & +\ C\, 2^k\prod_{i=1}^k\big(1+ \frac{1}{3p_i}\big)\, \min \big(B\big(\frac{3}{2}\big)^{k  } , 2^k\big) 
     . 
\end{align*}
Here $C$ is a universal constant.

 Also, 
\begin{eqnarray*} 
     N(B,F)&\ge &  2^{k}B+2^{k}\, B \sum_{ n\in F\atop 1<n\le B}
 \frac{2^{-\o(n)}  }{n}
 . \end{eqnarray*} 
Further,  if $F\subset [1,B]$,  \begin{eqnarray*}
     N(B,F)&\ge& 2^{k}B+2^{k} B \Big(\prod_{j=1}^k\big(1 +\frac{1}{2p_j}\big)\Big).
    \end{eqnarray*} 
 \end{theorem}
In the course of the proof, we show a better but less explicit result (see \eqref{estN(B,F)}), from which it follows that 
when the $p_i$ are all large, then   
\begin{align}      
     N(B,F) \sim B2^k\sum_{ n, m\in F\cap [1,B]}  \frac{1}{2^{\o(m)}2^{\o(n)}(n\vee m)} .
    \end{align}
  See Remark \ref{r1}. 

\vskip 3 pt    The paper is organized as follows. 
In the next section, some preparatory lemmas are established. In Sections \ref{s3} and \ref{s4} we prove Theorems \ref{nf} and \ref{np} respectively.  In
Section \ref{s5}, we conclude  with a remark concerning equation \eqref{e.k+1}, and give 
an elementary proof of an almost optimal upper bound of $N_{k}(B)$. We also suggest  a possible extension of Theorem \ref{np}.

 \section{\bf Proof of Proposition \ref{n}.} \label{s2}


    Equation \eqref{e2} means that $n_1n =n_3m$,
where $n_2=dn$, $n_4=dm$, $d=\gcd(n_2,n_4)$ and $\gcd(n,m)=1$. Since  $F$ is   factor closed, we necessarily have that $n,m\in F$.
\vskip 2 pt
 Now given $n,m\in F$ fixed  such that $\gcd(n,m)=1$, the  number of integers  $n_2, n_4\in F$ such that   $n_2=dn$, $n_4=dm$ for some $d\ge 1$,  is  obviously equal to
$$ \sum_{  n_2, n_4\in F\atop \frac{n_4}{n_2}=\frac{m}{n}} 1.$$
Further, the solutions to the equation $n_1n =n_3m$ in the unknowns $1\le n_1\le B$, $1\le n_3\le B$ are trivially $n_1=\l m$, $n_3 = \l n$, with $1\le \l\le \lfloor \frac{B}{n\vee m}\rfloor$ if ${n\vee m}\le B$, and there is no solution otherwise.
\vskip 2 pt

Hence the number of solutions in the unknowns $ n_1,n_3\in[1,B]$,  $n_2, n_4\in F$  verifying $\frac{n_1}{n_3}= \frac{n_4}{n_2}= \frac{m}{n}$, is 
   \begin{eqnarray}\label{e0}\Big( \sum_{  n_2, n_4\in F\atop \frac{n_4}{n_2}=\frac{m}{n}} 1\Big) \Big\lfloor \frac{B}{n\vee m}\Big\rfloor.
     \end{eqnarray}
Note that when $F=[1,B]$, this simplifies and one gets
   $  \big\lfloor \frac{B}{n\vee m}\big\rfloor^2$,  which for $n=m=1$ reduces to $B^2$. Also (see \cite[(4)]{S}),
 \begin{eqnarray*}  N(B)&=&\sum_{ 1\le n,m\le B\atop \gcd(n,m)=1}
  \Big\lfloor \frac{B}{n\vee m}\Big\rfloor^2
  \ =\    B^2 +2\sum_{ 1\le n<m\le B\atop \gcd(n,m)=1} \Big\lfloor \frac{B}{n\vee m}\Big\rfloor^2.
    \end{eqnarray*}
    In our case we get 
  \begin{eqnarray} \label{nbf}
   N(B, F )&=&\sum_{ n, m\in F\atop \gcd(n,m)=1}\Big( \sum_{  n_2, n_4\in F\atop \frac{n_4}{n_2}= \frac{m}{n}} 1\Big)\   \Big\lfloor \frac{B}{n\vee m}\Big\rfloor  .
    \end{eqnarray}
   If $F\subset [1,B]$, we have the obvious bound
$$ \sum_{  n_2, n_4\in F\atop \frac{n_4}{n_2}= \frac{m}{n}} 1\ \le \  \Big\lfloor \frac{B}{n\vee m}\Big\rfloor,$$
and so,
 \begin{eqnarray*}  N(B,F)&\le &\sum_{ n, m\in F\atop \gcd(n,m)=1}  \Big\lfloor \frac{B}{n\vee m}\Big\rfloor^2
 \ =\  B^2 +2\sum_{{ n, m\in F\atop n < m}\atop \gcd(n,m)=1} \Big\lfloor \frac{B}{ m}\Big\rfloor^2.
    \end{eqnarray*}
     Plainly,
  \begin{eqnarray*}
  \sum_{{ n, m\in F\atop n < m}\atop \gcd(n,m)=1} \Big\lfloor \frac{B}{ m}\Big\rfloor^2
  &\le & B^2
  \sum_{m\in F}\frac{1}{m^2}\sum_{{ n\in F\atop n < m}\atop \gcd(n,m)=1}1 \ \le\   B^2
  \sum_{m\in F
  }\frac{\phi(m)}{m^2} \ \le \   B^2
  \sum_{m\in F
  }\frac{1}{m } ,
    \end{eqnarray*}
since $\phi(m)\le m$.    Hence the claimed bound.

Further, 
  \begin{eqnarray} \label{nbf1}
   N(F, F )&=&\sum_{ n, m\in F\atop \gcd(n,m)=1}\Big( \sum_{  n_2, n_4\in F\atop \frac{n_4}{n_2}= \frac{m}{n}} 1\Big)^2  .
    \end{eqnarray}
 
 \vskip 2 pt 
 
 Next we prove the lower bound for $N(B,F)$ 
  in Proposition \ref{n}.
Let $(\nu,\m)\in F^2$ and write $\m=\gcd(\nu,\m)m$, $\nu=\gcd(\nu,\m)n$ with  $\gcd(n,m)=1$.  
  Associate to $(\nu,\m)$ the  set  $c(\nu,\m)=\big\{ \big((dm), \nu, (dn) ,\m\big),1\le d\le \gcd(\nu,\m)\big\}$. These quadruples provide $\gcd(\nu,\m)$ solutions of the restricted equation \eqref{e2}, since obviously   $ dm \nu\,=\, dn \m$ and $d\max( m ,n)\le \max(\nu,\m)\le B$.
\vskip 2 pt
 Naturally if  $(\m', \nu')\in F^2$ and  $(\m', \nu')\neq (\m, \nu)$, then $c(\nu',\m')\cap c(\nu,\mu)=\emptyset$.
Thus,
\begin{eqnarray*} N(B,F) \ \ge \  \sum_{\m,\nu\in F\atop } \gcd(\m,\nu).
 \end{eqnarray*}

 \begin{remark}  Let $F$  as in \eqref{F}. Then,
\begin{equation}\sum_{\m,\nu\in F\atop } \gcd(\m,\nu)\ =\  3^{k}\prod_{i=1}^k\Big(1+ \frac{1}{3}\,p_i\Big).
\end{equation} 
 Further,
  \begin{eqnarray}      N(F,F)&=& 6^k
 .  \end{eqnarray}
   Indeed, let $a\in F$, $n,m\in F$ with $\gcd(n,m)=a$. Thus $m=m_1a$, $n=n_1a$ with $\gcd(n_1,m_1)=1$. Because of the choice of $F$, we moreover have that $\gcd(n_1,a)=1=\gcd(m_1,a)$. Thus
\begin{eqnarray*}
  \sum_{\m,\nu\in F\atop } \gcd(\m,\nu)&=&\sum_{a\in F}a\sum_{m_1\in F\atop \gcd(m_1,a)=1}\sum_{{n_1\in F\atop \gcd(n_1,a)=1}\atop \gcd(n_1,m_1)=1}1\ =\ \sum_{a\in F}a\sum_{m_1\in F\atop
\gcd(m_1,a)=1}2^{k-\o(a)-\o(m_1)}
  \cr &=& 2^{k}\sum_{a\in F}a2^{-\o(a)}\sum_{j=0}^{k-\o(a)}C_{k-\o(a)}^j2^{-j}\ =\ 2^{k}\sum_{a\in F}a\, 2^{-\o(a)}\Big(\frac{3}{2}\Big)^{k-\o(a)}
  \cr &=&3^{k}\sum_{a\in F}a\Big(\frac{1}{3}\Big)^{\o(a)}\ =\  3^{k}\prod_{i=1}^k\Big(1+ \frac{1}{3}\,p_i\Big).
    \end{eqnarray*}
  
Further
\begin{eqnarray*}  \label{estexactFF}     N(F,F)&=&\sum_{\gcd(n,m)=1\atop n, m\in F}\Big( \sum_{  n_2, n_4\in F\atop \frac{n_4}{n_2}= \frac{m}{n}} 1\Big)^2\cr &=&\sum_{\gcd(n,m)=1\atop n, m\in F}   4^{k-\o(m)-\o(n)}\ =\ 5^{k}\sum_{m\in F}5^{-\o(m)}\ =\ 6^k
 .  \end{eqnarray*}
  \end{remark}

\section{\bf Proof of Theorem \ref{nf}.} \label{s3}

   By \eqref{nbf}, we have
  \begin{eqnarray*}
    N(B, F )&=&\#( F)\, B+2\sum_{\gcd(n,m)=1,n< m\atop   n, m\in F}\Big\lfloor\frac{B}{ m}\Big\rfloor \sum_{  n_2, n_4\in F\atop \frac{n_4}{n_2}= \frac{m}{n}} 1\ :=\ \Big(\# F\Big)B+2\mathcal B_0.
     \end{eqnarray*}
   We get
   \begin{eqnarray*}
    \mathcal B_0&=&\Big(\sum_{1< m\le  B^{\a}\atop    m\in F}+\sum_{ B^{\a} <m\le B^{2\a}\atop    m\in F}\Big)\Big\lfloor\frac{B}{ m}\Big\rfloor\sum_{\gcd(n,m)=1\atop n<m, n\in F}\sum_{  n_2, n_4\in F\atop \frac{n_4}{n_2}= \frac{m}{n}} 1\ :=\  \mathcal B_{01}+\mathcal B_{02}.
     \end{eqnarray*}
    Let us consider the first sum above, $\mathcal B_{01}$.
   \begin{eqnarray*}
    \mathcal B_{01}&=&\sum_{1< m\le  B^{\a}}\Big\lfloor\frac{B}{ m}\Big\rfloor \sum_{\gcd(n,m)=1, n<m} \Big\lfloor B^{\a}\Big\rfloor\ \le B^{1+\a}\sum_{ m\le B^{\a}} \frac{\phi(m)}{ m}\\
    &=& \frac6{\pi^2}\,  B^{1+2\a}+\mathcal O\Big(B^{1+\a}\log B\Big),
     \end{eqnarray*}   where    $ \sum_{n\le x}\frac {\phi(n)}n=\frac6{\pi^2}x+\mathcal O(\log x)$ is used.

    Now let us estimate $\mathcal B_{02}$.
    \begin{eqnarray*}
    \mathcal B_{02}&\le&\sum_{ B^{\a} <m\le B^{2\a}\atop    m\in F}\Big\lfloor\frac{B}{ m}\Big\rfloor\frac {B^{2\a}}m\sum_{n<m, (n, m)=1,\atop n\in F}1\ \le \  B^{1+2\a}\sum_{  B^{\a} <m\le B^{2\a}\atop    m\in F}\frac {\phi(m)}{m^2}\\
                   &\le& B^{1+2\a}\sum_{  B^{\a} <m\le B^{2\a}\atop    m\in F}\frac 1{m}.
   \end{eqnarray*}  Putting together the above estimates, we have 
\begin{eqnarray*}
    N(B, F )&\le &\Big(1+\frac {12}{\pi^2}+2\sum_{  B^{\a} <m\le B^{2\a}\atop    m\in F}\frac 1{m}\Big)B^{1+2\a}.
     \end{eqnarray*}

\section{Proof of Theorem \ref{np}.} \label{s4} 
 
  By Lemma \ref{n}, 
    \begin{eqnarray}  \label{estexact}    
     N(B,F)&=&\sum_{ n, m\in F\atop \gcd(n,m)=1}\Big( \sum_{  n_2, n_4\in F\atop \frac{n_4}{n_2}= \frac{m}{n}} 1\Big)\  
 \Big\lfloor \frac{B}{n\vee m}\Big\rfloor
 . \end{eqnarray}

   Let $m\in F$.      Given an integer $a\ge 2$, we define $\langle a\rangle= \{p: p|a\}$.  Recall that $\o(n)$ denotes the prime divisor function, and $\o(1)=0$. Consider  for $n,m\in F$ with $\gcd(n,m)=1$, the sum
    $$ \sum_{  n_2, n_4\in F\atop \frac{n_4}{n_2}= \frac{m}{n}} 1.$$
 Then $ \frac{n_4}{n_2}= \frac{m}{n}$ gives rise to solutions
 $n_4=\l m$,  $n_2= \l n$.    As  $\l$, $\l m$, $\l n\in F$,
 it follows by definition of $F$ that $\langle \l\rangle\cap \langle m\rangle=\langle \l\rangle\cap \langle n\rangle=\emptyset$. Otherwise, if for
instance some $p_j$ verifies $p_j\in \langle \l\rangle\cap \langle m\rangle$, then $p_j^2|\l m=n_4$, which is impossible. Thus $\langle \l\rangle\subset
\{p_1, \ldots, p_k\}-  \langle m\rangle- \langle n\rangle$. Conversely any subset $A$ of it provides  a suitable $\l$ with $\langle \l\rangle=A$.  And
so  we have
\begin{eqnarray*}   \sum_{  n_2, n_4\in F\atop \frac{n_4}{n_2}= \frac{m}{n}} 1 
&=& 2^{k-\o(m)-\o(n)}.
 \end{eqnarray*}
 for all $n,m\in F$ with $\gcd(n,m)=1$. 
 
  \vskip 3 pt Inserting this into \eqref{estexact}  we get, 
    \begin{eqnarray}  \label{estexacta}    
     N(B,F)
   &=&  \sum_{ n, m\in F\atop \gcd(n,m)=1}2^{k-\o(m)-\o(n)} \Big\lfloor \frac{B}{n\vee m}\Big\rfloor. 
\end{eqnarray}
 
First consider the lower bound. We have
\begin{eqnarray*} 
     N(B,F)&\ge &2^{k}\sum_{ n\in F\atop n\le B}2^{-\o(n)} \  
 \Big\lfloor \frac{B}{n}\Big\rfloor  . \end{eqnarray*}
 Observe that for  $X\neq 0$,
  \begin{eqnarray}\label{X}
    \sum_{  d\in F} \frac{1}{d}X^{- \o(d)}&=&\prod_{j=1}^k\Big(1 +\frac{1}{Xp_j}\Big) .
    \end{eqnarray}  
  Thus  if $F\subset [1,B]$,  \begin{eqnarray*}
     N(B,F)&\ge& 2^{k} B \Big(\prod_{j=1}^k\big(1 +\frac{1}{2p_j}\big)\Big),
    \end{eqnarray*} 
which proves the lower bound. 
    
    \vskip 3 pt Next consider the upper bound for $N(B,F)$.
 We have 
 \begin{eqnarray}  \label{estexacta1}    
     N(B,F)
  &=& B\, 2^{k}\sum_{{ n, m\in F\atop \gcd(n,m)=1}\atop (m\vee n)\le B}  \frac{1}{2^{\o(m)}2^{\o(n)}(n\vee m)}
 \cr & & \   +\  2^k\,\mathcal O\Big(
\sum_{{ n, m\in F\atop \gcd(n,m)=1}\atop (m\vee n)\le B} 
\frac{1}{2^{\o(m)}\,2^{\o(n)} }\Big). 
\end{eqnarray}
   
     Put 
\begin{eqnarray*} Y=  \sum_{{ n, m\in F\atop\gcd(n,m)=1}\atop (m\vee n)\le B}  \frac{1}{2^{\o(m)}2^{\o(n)}(n\vee m)} , \qq \qq  Y_0=\sum_{{ n, m\in F\atop \gcd(n,m)=1}\atop (m\vee n)\le B} 
\frac{1}{2^{\o(m)}\,2^{\o(n)} } .
 \end{eqnarray*}

 We thus start with the formula
 \begin{eqnarray}  \label{estexact0}    
     N(B,F)&=&  B\, 2^{k}\, Y + 2^k\,\mathcal O (
Y_0 ). \end{eqnarray}
  We note that 
 \begin{eqnarray*} Y&=& 1  + 2\,\sum_{   m\in F   \atop   m\le B}\frac{1}{m\,2^{\o(m)}  } \sum_{{ n \in F\atop\gcd(n,m)=1}  \atop n< m } 
\frac{1}{ 2^{\o(n)} }  .\end{eqnarray*}

  The  presence of the order relation 
  \lq\lq$<\,$\rq\rq on $F$,  a set of squarefree numbers, in the summation index, makes that sum  not easy to manipulate. We cannot   bound $Y$ directly
and will  thus proceed   differently.
    We first note the relation 
   \begin{eqnarray}\label{vee} \frac{1}{n\vee m}=\frac{1}{\sqrt n}\, \frac{1}{\sqrt m}\, \Big(\frac{n\wedge m}{n\vee m} \Big)^{1/2}.
 \end{eqnarray} 
 \vskip 3 pt
 Now as $ e^{-|\t|}= \int_\R  e^{i\t t}\frac{\dd t}{\pi( t^2+1)}$,
it follows that
\begin{eqnarray}  \Big(\frac{n }{m}\Big)^s=\int_\R
 \frac{1}{  n^{- ist}m^{ist}}\frac{\dd t}{\pi( t^2+1)} \qq \qquad (m\ge n).
\end{eqnarray}
Take $s=1/2$. We get
$$\frac{1}{n\vee m}=\frac{1}{\sqrt n}\frac{1}{\sqrt m}\int_\R
 \frac{1}{  n^{- it/2}m^{i t/2}}\frac{\dd t}{\pi( t^2+1)}.$$
Recall that $\m$ denotes the M\"obius function and that
\begin{eqnarray} \label{sp} \sum_{d|n}\m(d)=\d(n):=\begin{cases}1\quad & {\rm if}\ n=1,\cr
0\quad & {\rm if}\ n \not=1. \end{cases} \end{eqnarray}
Putting this together, we have
\begin{eqnarray}    Y&=&   \sum_{{ n, m\in F\atop \gcd(n,m)=1}\atop (m\vee n)\le B}   \frac{1}{2^{ \o(n)}\sqrt n}\frac{1}{2^{ \o(m) }\sqrt m}\int_\R
 \frac{1}{  n^{- it/2}m^{i t/2}}\frac{\dd t}{\pi( t^2+1)}
 \cr &=&   \sum_{d\in F} \m(d)\sum_{ { n, m\in F\atop d|n, d|m}\atop (m\vee n)\le B}     \frac{1}{2^{ \o(n)}\sqrt n}\frac{1}{2^{ \o(m) }\sqrt m}\int_\R
 \frac{1}{  n^{- it/2}m^{i t/2}}\frac{\dd t}{\pi( t^2+1)}
\cr &=&   \sum_{d\in F}\m(d) \int_\R\Big|\sum_{ { n\in F\atop d|n}\atop  n\le B}      \frac{1}{2^{ \o(n)}\sqrt n}  \frac{1}{  n^{  it/2}
}\Big|^2\frac{\dd t}{\pi( t^2+1)}
 .
    \end{eqnarray}

By the very definition of $F$, if $ n \in F$ and $d|n $, then $n=\nu d$ and $\langle \nu\rangle \subset \{p_1, \ldots, p_k\}-  \langle d\rangle $. 
Thus
$$\sum_{ { n\in F\atop d|n}\atop  n\le B}      \frac{1}{2^{ \o(n)}\sqrt n}  \frac{1}{  n^{  it/2} }=\frac{1}{2^{ \o(d)}  d^{   (1+it)/2}}\sum_{ { \nu \in F\atop \gcd(\nu , d)=1}\atop \nu\le B/d }     \frac{1}{2^{ \o(\nu)}\sqrt \nu}  \frac{1}{  \nu^{  it/2} }.$$
So that,
\begin{eqnarray*}    Y
&=&  \int_\R\Big|\sum_{ { \nu \in F}\atop \nu\le B}     \frac{1}{2^{ \o(\nu)}\sqrt \nu}  \frac{1}{  \nu^{  it/2} }\Big|^2\frac{\dd t}{\pi( t^2+1)}
\cr &&  +\  \sum_{d\in F\atop d\neq 1}\frac{\m(d)}{2^{ \o(d)}  d} \ \int_\R\Big|\sum_{ { \nu \in F\atop \gcd(\nu , d)=1}\atop \nu\le B/d }     \frac{1}{2^{ \o(\nu)}\sqrt \nu}  \frac{1}{  \nu^{  it/2} }\Big|^2\frac{\dd t}{\pi( t^2+1)}.
  \end{eqnarray*}
We notice that 
$$\int_\R\Big|\sum_{ { \nu \in F}\atop \nu\le B}     \frac{1}{2^{ \o(\nu)}\sqrt \nu}  \frac{1}{  \nu^{  it/2} }\Big|^2\frac{\dd t}{\pi( t^2+1)}=\sum_{
n, m\in F\cap [1,B]}   \frac{1}{2^{\o(m)}2^{\o(n)}(n\vee m)}.$$ By Lemma
2.4 in \cite{W2},
  \begin{eqnarray*} \int_\R\Big|\sum_{n=1}^Nx_n n^{ i s t}\Big|^2
 \frac{\dd t}{\pi( t^2+1)} \, =\,
  \sum_{j=1}^N (j^{2s}-(j-1)^{2s})\Big|\sum_{\m=j}^N  \frac{x_\mu}{\mu ^s} \Big|^2
\end{eqnarray*}
for any real $s\ge 0$. In our case $s=1/2$. Now 
by Lemma 2.5 in \cite{W2},
  \begin{eqnarray*} \sum_{j=1}^N (j^{2s}-(j-1)^{2s})\Big|\sum_{\m=j}^N  \frac{x_\mu}{\mu ^s} \Big|^2&\le &  \begin{cases}C_s\sum_{\m=1}^N
 {|x_\m|^2} \m^{  3/2-2s} &\quad {\rm if}\  0<s<1/4, \cr
C\sum_{\m=1}^N
 |x_\m|^2  \m \log \m&\quad {\rm if}\  s= 1/4,\cr
C_s\sum_{\m=1}^N
 |x_\m|^2  \m &\quad {\rm if}\  s> 1/4   ,
\end{cases}\end{eqnarray*}
for any $s>0$ and complex numbers   $x_j$, $j=1,\ldots , N$.
\vskip 2 pt
 Therefore,
\vskip 2 pt
-- If $d=1$, we have
  \begin{eqnarray*}  \int_\R\Big|\sum_{  \nu \in F \atop \nu\le B}     \frac{1}{2^{ \o(\nu)}\sqrt \nu}  \frac{1}{  \nu^{  it/2} }\Big|^2\frac{\dd t}{\pi( t^2+1)}
&\le &C \sum_{  \nu \in F\atop \nu\le B  }     \frac{1}{2^{ 2\o(\nu)}  \nu}\ \nu\ =\ \sum_{  \nu \in F \atop \nu\le B }     \frac{1}{4^{  \o(\nu)} }.\end{eqnarray*}
We note that 
$$\sum_{  \nu \in F \atop \nu\le B }     \frac{1}{4^{  \o(\nu)} }\le B\sum_{  \nu \in F }     \frac{1}{\nu\, 4^{  \o(\nu)} }=B
\sum_{p_{i_1},\ldots,p_{i_r}\atop 1\le r\le k} \frac{1}{p_{i_1}\ldots p_{i_r} 4^{ k} }=B\prod_{i=1}^k \Big(1+\frac{1}{4 p_i}\Big) $$
Further,
 $$\sum_{  \nu \in F }     \frac{1}{4^{  \o(\nu)}   } =\sum_{y=0}^{k }\sum_{\nu \in F, \o(\nu)=y}4^{-y}= \sum_{y=0}^{k
}C^y_k4^{-y}=\Big(\frac{5}{4}\Big)^{k  }. $$

 Thus 
\begin{align}\label{d=1} \int_\R\Big|\sum_{  \nu \in F \atop \nu\le B}     \frac{1}{2^{ \o(\nu)}\sqrt \nu}  \frac{1}{  \nu^{  it/2} }\Big|^2\frac{\dd t}{\pi( t^2+1)}
   \le    \min\Big(B\prod_{i=1}^r \Big(1+\frac{1}{4 p_i}\Big)\,,\,  \Big(\frac{5}{4}\Big)^{k  } \Big).
    \end{align}
 
-- If $d>1$, $d\in F$, then similarly,
 \begin{eqnarray*} & &\int_\R\Big|\sum_{{  n \in F\atop d|n }\atop n\le B}     \frac{1}{2^{ \o(n)}\sqrt n}  \frac{1}{  n^{  it/2} }\Big|^2\frac{\dd
t}{\pi( t^2+1)}
\cr &=& \frac{1}{2^{2 \o(d)}  d }\int_\R\Big|\sum_{ { \nu \in F\atop \gcd(\nu , d)=1}\atop \nu\le B/d }     \frac{1}{2^{ \o(\nu)}\sqrt \nu} 
\frac{1}{  \nu^{  it/2} }\Big|^2\frac{\dd t}{\pi( t^2+1)}
 \cr &\le & \frac{C}{2^{2 \o(d)}  d } \sum_{ { \nu \in F\atop \gcd(\nu , d)=1}\atop \nu\le B/d }       \frac{1}{2^{ 2\o(\nu)}\nu  }\cdot  \nu
  \ =\ \frac{C}{2^{2 \o(d)}  d } \sum_{ { \nu \in F\atop \gcd(\nu , d)=1}\atop \nu\le B/d }    \frac{1}{4^{ \o(\nu)}   }.
 \end{eqnarray*}
We also note that 
$$\sum_{ { \nu \in F\atop \gcd(\nu , d)=1}\atop \nu\le B/d } \frac{1}{4^{ \o(\nu)}   }\le \frac{B}{d}\sum_{ { \nu \in F\atop \gcd(\nu , d)=1} }    \frac{1}{\nu 4^{ \o(\nu)}   }= \frac{B}{d}\prod_{1\le i\le k\atop p_i\not|\, d} \Big(1+\frac{1}{4 p_i}\Big).$$

Next,
\begin{eqnarray*} \sum_{ { \nu \in F\atop \gcd(\nu , d)=1} }    \frac{1}{4^{ \o(\nu)}   }&=&\sum_{y=0}^{k-\o(d)} \sum_{\nu\in F, \o(\nu)=y}4^{-y} \ =\  \sum_{y=0}^{k-\o(d) }C^y_{k-\o(d)}4^{-y} \ =\ \Big(\frac{5}{4}\Big)^{k -\o(d) }
\cr &=&  \Big(\frac{5}{4}\Big)^k \Big(\frac 45\Big)^{  \o(d)}    .
\end{eqnarray*}
Therefore, 
\begin{align*} \int_\R\Big|\sum_{{  n \in F\atop d|n }\atop \nu\le B}     \frac{1}{2^{ \o(n)}\sqrt n}  \frac{1}{  n^{  it/2} }\Big|^2&\frac{\dd
t}{\pi( t^2+1)}
\cr &\le  \ \frac{C}{4^{ \o(d)}  d } \min\bigg(\frac{B}{d}\prod_{1\le i\le k\atop p_i\not|\, d} \Big(1+\frac{1}{4 p_i}\Big),\Big(\frac{5}{4}\Big)^k \Big(\frac 45\Big)^{  \o(d)}\bigg).
 \end{align*}

Consequently,
\begin{align*}   \sum_{d\in F\atop
d\neq 1}|\m(d) | \int_\R&\Big|\sum_{ { n\in F\atop d|n}\atop  n\le B}      \frac{1}{2^{ \o(n)}\sqrt n}  \frac{1}{  n^{  it/2}
}\Big|^2\frac{\dd t}{\pi( t^2+1)}
\cr 
 \le & \sum_{d\in F\atop
d\neq 1}\frac{C}{4^{ \o(d)}  d } \min\bigg(\frac{B}{d}\prod_{1\le i\le k\atop p_i\not|\, d} \Big(1+\frac{1}{4 p_i}\Big),\Big(\frac{5}{4}\Big)^k
 \Big(\frac 45\Big)^{  \o(d)}\bigg).
 \end{align*}
On the one hand, 

\begin{eqnarray*}  \sum_{d\in F\atop
d\neq 1}\frac{1}{4^{ \o(d)}  d } \Big(\frac{5}{4}\Big)^k \Big(\frac 45\Big)^{  \o(d)}&=& \Big(\frac{5}{4}\Big)^k\sum_{d\in F\atop
d\neq 1}\frac{1}{5^{ \o(d)}  d } \ \le\  \Big(\frac{5}{4}\Big)^k\prod_{i=1}^k\Big(1+\frac{1}{5p_i}\Big).
    \end{eqnarray*}
On the other hand, using the definition of $F$,

\begin{eqnarray*}
& &B\sum_{d\in F\atop
d\neq 1}\frac{1}{4^{ \o(d)}  d^2}\prod_{1\le i\le k\atop p_i\not|\, d} \Big(1+\frac{1}{4 p_i}\Big)
 \cr&=& C\, B\prod_{i=1}^k\Big(1+\frac{1}{4p_i}\Big)\sum_{d\in F\atop
d\neq 1}\frac{C}{4^{ \o(d)}  d^2\prod_{1\le i\le k\atop p_i|\, d} (1+{1}/{4 p_i})}
\cr &=& C\, B\prod_{i=1}^k\Big(1+\frac{1}{4p_i}\Big)\sum_{d\in F\atop
d\neq 1}\prod_{1\le i\le k\atop p_i|\, d}\frac{1}{ (4p_i^2+4 p_i)}
\cr&= &  C\, B\prod_{i=1}^k\Big(1+\frac{1}{4p_i}\Big)\bigg[\prod_{i=1}^k\Big(1+\frac{1}{ (4p_i^2+4 p_i)}\Big)-1\bigg]
\cr&= &  C\, B\prod_{i=1}^k\Big(1+\frac{1}{4p_i}\Big)\e(F),\end{eqnarray*}
where we put
\begin{equation} \label{e(F)} \e(F)= \prod_{i=1}^k\Big(1+\frac{1}{ 4p_i( p_i +  1)}\Big)-1.
\end{equation} 
Thus
\begin{eqnarray*}& &   \sum_{d\in F\atop
d\neq 1}|\m(d) |\int_\R\Big|\sum_{ { n\in F\atop d|n}\atop  n\le B}      \frac{1}{2^{ \o(n)}\sqrt n}  \frac{1}{  n^{  it/2}
}\Big|^2\frac{\dd t}{\pi( t^2+1)}
\cr 
& \le & C\,\min\bigg(  B\prod_{i=1}^k\Big(1+\frac{1}{4p_i}\Big)\e(F), \Big(\frac{5}{4}\Big)^k\prod_{i=1}^k\Big(1+\frac{1}{5p_i}\Big)   \bigg),
    \end{eqnarray*}
whence,
  \begin{align*}\Big| Y- \sum_{ n, m\in F\cap [1,B]}&  \frac{1}{2^{\o(m)}2^{\o(n)}(n\vee m)}\Big|
\cr &\le  C\,\prod_{i=1}^k\big(1+\frac{1}{4p_i}\big)\min\Big(  B\,\e(F), \Big(\frac{5}{4}\Big)^k\   \Big) .
 \end{align*}
Using now \eqref{d=1}  
 we obtain the bound,
\begin{eqnarray*}  
  Y&\le &     \min\big(B\prod_{i=1}^k \big(1+\frac{1}{4 p_i}\big)\,,\,  \big(\frac{5}{4}\big)^{k  } \big)+ 
C\,\prod_{i=1}^k\big(1+\frac{1}{4p_i}\big)\min\big(  B\,\e(F), \big(\frac{5}{4}\big)^k\   \big)
\cr &\le &   C\, \prod_{i=1}^k \big(1+\frac{1}{4 p_i}\big) \min\big(B \,,    \big(\frac{5}{4}\big)^{k  } \big)  , 
    \end{eqnarray*} 
since $\e(F)\le C$ uniformly in $F$. Now plainly,
  \begin{eqnarray}  \label{Y0}   \sum_{(n,m)=1\atop n, m\in F}  \frac{1}{2^{\o(m)+\o(n)}}
  &=&  \sum_{ m\in F} \frac{1}{2^{\o(m)}}  \sum_{ n\in F\atop (n,m)=1}  \frac{1}{2^{\o(n)}}  
\ =\    \sum_{ m\in F} \frac{1}{2^{\o(m)}}\Big(\frac{3}{2}\Big)^{k-\o(m)  }
  \cr &=&  \Big(\frac{3}{2}\Big)^{k  }\sum_{ m\in F} \frac{1}{3^{\o(m)}} \ =\  \Big(\frac{3}{2}\Big)^{k  }\Big(\frac{4}{3}\Big)^{k  }\ =\ 2^{k  }. 
 \end{eqnarray}
Also
  \begin{eqnarray*} \sum_{{ n, m\in F\atop \gcd(n,m)=1}\atop (m\vee n)\le B} 
\frac{1}{2^{\o(m)}\,2^{\o(n)} }&\le & B\sum_{{ n, m\in F\atop \gcd(n,m)=1}  } \frac{1}{m 2^{\o(m)}\,2^{\o(n)} }\ = \ B\sum_{ m\in F} \frac{1}{m2^{\o(m)}} 
\sum_{ n\in F\atop (n,m)=1}  \frac{1}{2^{\o(n)}}
\cr &=& B\sum_{ m\in F} \frac{1}{m2^{\o(m)}} 
\Big(\frac{3}{2}\Big)^{k-\o(m)  }\ =\ B\Big(\frac{3}{2}\Big)^{k  }\sum_{ m\in F} \frac{1}{m3^{\o(m)}}
\cr &=& B\Big(\frac{3}{2}\Big)^{k  }\prod_{i=1}^k\Big(1+ \frac{1}{3p_i}\Big).
 \end{eqnarray*}
Thus
\begin{eqnarray*} Y_0&\le & \min \Big(B\big(\frac{3}{2}\big)^{k  }\prod_{i=1}^k\big(1+ \frac{1}{3p_i}\big), 2^k\Big) .
 \end{eqnarray*}

As by  \eqref{estexact}, $ N(B,F) =  B 2^{k}\, Y + 2^k\,\mathcal O (
Y_0 )$, we get,
\begin{eqnarray}  \label{estN(B,F)}    
  & & \Big|  N(B,F) -B2^k\sum_{ n, m\in F\cap [1,B]}  \frac{1}{2^{\o(m)}2^{\o(n)}(n\vee m)}\Big|
  \cr  &\le & \ 2^k\,\mathcal O \Big(
\min\big(B\prod_{i=1}^k \big(1+\frac{1}{4 p_i}\big)(1+\e(F))\,,\,   \big(\frac{5}{4}\big)^{k  } \big) \Big )
\cr &  & + \ C\, B 2^{k}
 \prod_{i=1}^k\big(1+\frac{1}{4p_i}\big)\min\big(  B\,\e(F), \big(\frac{5}{4}\big)^k\   \big).
 \end{eqnarray}
 Indeed,
 \begin{eqnarray*}     
 & &  \Big|  N(B,F) -B2^k\sum_{ n, m\in F\cap [1,B]}  \frac{1}{2^{\o(m)}2^{\o(n)}(n\vee m)}\Big|
 \cr & \le & \Big|  N(B,F) -B 2^{k} \,Y\Big| + B 2^{k}\Big| Y- \sum_{ n, m\in F\cap [1,B]}   \frac{1}{2^{\o(m)}2^{\o(n)}(n\vee m)}\Big|
\cr &\le & 2^k\,\mathcal O (
Y_0 ) + C B 2^{k}  \,\prod_{i=1}^k\big(1+\frac{1}{4p_i}\big)\min\Big(  B\,\e(F), \Big(\frac{5}{4}\Big)^k\   \Big)
\cr& \le & 2^k\,\mathcal O \Big(\min \Big(B\big(\frac{3}{2}\big)^{k  }\prod_{i=1}^k\big(1+ \frac{1}{3p_i}\big), 2^k\Big) \Big )
\cr &  & + C  B 2^{k}
 \prod_{i=1}^k\big(1+\frac{1}{4p_i}\big)\min\big(  B\,\e(F), \big(\frac{5}{4}\big)^k\   \big)
  . 
\end{eqnarray*}
 
Also
 
 \begin{eqnarray}  \label{estN(B,F)1}    
      N(B,F)&\le & B2^k \min\Big(B\prod_{i=1}^r \Big(1+\frac{1}{4
p_i}\Big)\,,\,  \Big(\frac{5}{4}\Big)^{k  } \Big)  +
  \cr  & & +\ 2^k\,\mathcal O \Big(\min \Big(B\big(\frac{3}{2}\big)^{k  }\prod_{i=1}^k\big(1+ \frac{1}{3p_i}\big), 2^k\Big)\Big )
\cr &  & +\ C\, B 2^{k}
 \prod_{i=1}^k\big(1+\frac{1}{4p_i}\big)\min\big(  B\,\e(F), \big(\frac{5}{4}\big)^k\   \big)
 \cr &\le & C\,B2^k\prod_{i=1}^k \big(1+\frac{1}{4
p_i}\big) \min\big(B \,,\,  \big(\frac{5}{4}\big)^{k  } \big)  +
  \cr  & & +\ C\, 2^k\prod_{i=1}^k\big(1+ \frac{1}{3p_i}\big)\, \min \Big(B\big(\frac{3}{2}\big)^{k  } , 2^k\Big) 
    . 
\end{eqnarray}

\begin{remark}\label{r1} When the $p_i$ are all large, then $\e(F) $ becomes small and we see with \eqref{estN(B,F)} that 
\begin{align}      
     N(B,F) \sim B2^k\sum_{ n, m\in F\cap [1,B]}  \frac{1}{2^{\o(m)}2^{\o(n)}(n\vee m)} .
    \end{align}
  \end{remark}

\begin{remark}By Weierstrass' inequality, if $0<a_k<1$ and $\sum_{k=1}^n a_k<1$, then
 $$1+\sum_{k=1}^n a_k <\prod_{k=1}^n \big(1+a_k\big)<\frac{1}{1-\sum_{k=1}^n a_k}.$$
See Mitrinovi\'c \cite[3.2.37(3)]{M}. Thus if $\sum_{i=1}^k\frac{1}{ p_i^2+ p_i}\le 2$, 
  \begin{eqnarray*} 
\e(F)= \prod_{i=1}^k\Big(1+\frac{1}{ (4p_i^2+4 p_i)}\Big)-1\le  \frac{ \sum_{i=1}^k\frac{1}{ (4p_i^2+4 p_i)}}{1-  \sum_{i=1}^k\frac{1}{ (4p_i^2+4 p_i)}}\
\le
\ 2 \sum_{i=1}^k\frac{1}{ p_i^2+ p_i}.
   \end{eqnarray*}
 \end{remark}

 \section{Concluding Remarks.}\label{s6}

 \subsection{A remark concerning Equation \ref{e.k+1}.}\label{s5}

   Granville and Soundarajan (unpublished) proved using contour
integral representation the following estimate
 \begin{equation}\label{e.k}N_k(B) \sim  c(k) B^{k+1} (\log B)^{k^2},\end{equation}
 where the constant $c(k)$ depends on $k$ only. 
We show here  that the following almost optimal upper bound
 \begin{equation}\label{t1a}
N_{k}(B) \ll_k   B^{k+1}\big(\log B\big)^{k^2+2k-2}, \qq\qq (k\ge 1), \end{equation}
can be proved quite elementarily. It is an interesting  question to know  whether the  approach we propose can be used to remove the extra term $2k-2$.

 Let   $d_k(n)$ denote the Piltz divisor function   counting the number of ways to write $n$ as a product of $k$ factors.
We will use the fact that   $d_k$ is sub-multiplicative: $ d_k(nm)\le d_k(n)d_k(m)$,  for all $n,m\ge 1$. 
  This follows from the formula (\cite{C}, p. 40)
$$d_k(n)= \prod_p   C_{v_p(n)+k-1}^{v_p(n)}= \prod_p \prod_{j=1}^{k-1} \Big(\frac{v_p(n)+j}{j}\Big),$$
where  $v_p(n)$ is the $p$-valuation of $n$, i.e. $p^{v_p(n)}||n$ and $v_p(1)\equiv 0$.

\begin{proof}[Proof of \eqref{t1a}]  Applying Lemma \ref{n}  with $F=\mathcal M_k(B)$ gives,
  \begin{eqnarray*}  N_{k}(B)&=&\sum_{\gcd(n,m)=1\atop 1\le n, m\le B}\,  \Big\lfloor \frac{B}{n\vee m}\Big\rfloor\,\Big( \sum_{  m_1, m_2\in \mathcal M_k(B)\atop \frac{m_2}{m_1}= \frac{n}{m}} 1\Big)
   .
    \end{eqnarray*}
 Thus
\begin{eqnarray*}  N_{k}(B)&=&B\#(\mathcal M_k(B)) + 2\sum_{\gcd(n,m)=1\atop 1\le n< m\le B}\,  \Big\lfloor \frac{B}{ m}\Big\rfloor\,\Big( \sum_{  m_1, m_2\in \mathcal M_k(B)\atop \frac{m_2}{m_1}= \frac{n}{m}} 1\Big)
 \cr & :=& B\#(\mathcal M_k(B)) + 2\mathcal B_k .
    \end{eqnarray*}
We note that $\frac{m_2}{m_1}= \frac{n}{m}$ for some $m_1, m_2\in \mathcal M_k(B)$, means that $mm_2=c=nm_1$, $nm|c$ and $nm\le c\le nB^k$.
Thus
\begin{eqnarray}\label{mcdk}
 \sum_{  m_1, m_2\in \mathcal M_k(B)\atop \frac{m_2}{m_1}= \frac{n}{m}} 1
&= &\sum_{{nm_{1,1}\ldots m_{1,k}=mm_{2,1}\ldots m_{2,k}\atop 1\le m_{1,i}, m_{2,j}\le B}\atop 1\le i,j\le k} 1
\cr &=& \sum_{c=nm\atop nm|c }^{nB^k}
                              \Big(
                              \sum_{{nm_{1,1}\ldots m_{1,k}=c\atop 1\le m_{1,i}\le  B}\atop 1\le i\le k }1\Big) \Big(
                              \sum_{
                              {mm_{2,1}\ldots m_{2,k}=c\atop 1\le m_{2,i}\le  B}\atop 1\le i\le k}1\Big)
                           \cr 
    (c=jmn)\qq                       &=& \sum_{j=1}^{[   {B^k}/m ]}
                              \Big(
                              \sum_{{m_{1,1}\ldots m_{1,k}=jm\atop 1\le m_{1,i}\le  B}\atop 1\le i\le k}1\Big) \Big(
                              \sum_{{m_{2,1}\ldots m_{2,k}=jn\atop 1\le m_{2,i}\le  B}\atop 1\le i\le k}1\Big),
\end{eqnarray}
and so,
\begin{eqnarray}\label{mcdk1}
 \sum_{  m_1, m_2\in \mathcal M_k(B)\atop \frac{m_2}{m_1}= \frac{n}{m}} 1
              &\le &\sum_{j=1}^{[   {B^k}/m ]}d_k(jn)d_k(jm)
\ \le \ d_k(n)d_k(m) \sum_{j=1}^{ [   {B^k}/m ]}d_k(j)^2
\cr &\ll_k & \  \frac {B^k}m \big(\log B\big)^{k^2-1}   d_k(n)d_k(m),
\end{eqnarray}
where we have used sub-multiplicativity of $d_k$ and the estimate
 $\sum_{m\le x} {d_{ k}^2(m)}= (C_k +o(1))x\log^{k^2-1} x$.
   See  \cite[(9.33)]{I} for instance.
Thus \begin{eqnarray}\label{bk}
 \mathcal B_k\ll_k B^{k+1}\big(\log B\big)^{k^2-1} \sum_{2\le m\le B}\sum_{n<m\atop{\gcd(n,m)=1}}\frac {d_k(n)d_k(m)}{m^2}.
 \end{eqnarray}
 Now, plainly
 \begin{eqnarray*}
  \sum_{2\le m\le B}\sum_{n<m\atop{\gcd(n,m)=1}}\frac {d_k(n)d_k(m)}{m^2} & \le & \sum_{2\le m\le B}\frac { d_k(m)}{m^2}\sum_{n<m }d_k(n) \cr & \ll_k& (\log  B)^{k-1}\sum_{2\le m\le B}\frac { d_k(m)}{m }
 \cr &\ll_k & (\log  B)^{2k-1}  ,
\end{eqnarray*}
 since $\sum_{m\le x} d_{ k}(m)\sim  C_k x (\log  x)^{k-1}$, (see notably Theorem 14.9 in \cite{I}),
and further that $ \sum_{n\le x} \frac {d_k(n)}n\sim C_k (\log x)^{k}$.
This along with \eqref{bk} implies
 \begin{eqnarray}\label{bk1}
 \mathcal B_k&\ll_k& B^{k+1}\big(\log B\big)^{k^2+2k-2}.
 \end{eqnarray}
  By combining and since  $\#\big(\mathcal M_k(B)\big)= B^k$,
    $$N_k(B)\ll_k B^{k+1}{\big(\log B\big)}^{k^2+2k-2}.$$
\end{proof} 

 \subsection{Problem.}\label{s5}
 Consider Equation \eqref{e2} with $F=\{n\le B: n\ \hbox{squarefree}\}$. For the study of this very interesting case,   part of the proof of Theorem 1.3   can probably still be used. This will be investigated elsewhere.


\bigskip

\noindent {\bf Acknowledgement.} We wish to thank the anonymous referee for a thorough reading of the paper and helpful remarks.


\end{document}